\documentclass[11 pt]{amsart}
\usepackage{amsmath,amsthm,amsfonts,amssymb,latexsym, mathrsfs}
\usepackage{hyperref}
\usepackage{color}

\theoremstyle{plain}

\newtheorem{theorem}{Theorem}[section]
\numberwithin{equation}{section} 

\numberwithin{figure}{section} 

\theoremstyle{plain}

\theoremstyle{plain}

\theoremstyle{plain}

\newtheorem{claim}[theorem]{Claim}
\theoremstyle{plain}

\theoremstyle{plain}

\theoremstyle{plain}

\newtheorem{exm}[theorem]{Example}
\theoremstyle{plain}

\newtheorem{lemma}[theorem]{Lemma}
\theoremstyle{plain}

\theoremstyle{plain}

\theoremstyle{plain}

\theoremstyle{plain}

\newtheorem{rmr}[theorem]{Remark}
\theoremstyle{plain}

\theoremstyle{plain}

\newtheorem*{theorem*}{Theorem}

\theoremstyle{definition}
\newtheorem{dfn}[theorem]{Definition}
\theoremstyle{definition}

\newtheorem{definition}[theorem]{Definition}

\def\Ind#1#2{#1\setbox0=\hbox{$#1x$}\kern\wd0\hbox to 0pt{\hss$#1\mid$\hss}
\lower.9\ht0\hbox to 0pt{\hss$#1\smile$\hss}\kern\wd0}
\def\Notind#1#2{#1\setbox0=\hbox{$#1x$}\kern\wd0\hbox to 0pt{\mathchardef\nn="0236\hss$#1\nn$\kern1.4\wd0\hss}\hbox to 0pt{\hss$#1\mid$\hss}\lower.9\ht0
\hbox to 0pt{\hss$#1\smile$\hss}\kern\wd0}

\newcommand{\G}{\mathcal{G}}

\DeclareMathOperator{\Def}{Def}

\title{A fixed-point theorem for definably amenable groups}

\author{Juan Felipe Carmona}
\address{Universidad Antonio Nari\~no, Calle 58 A No. 37 - 94, Bogot\'{a}, Colombia}
\email{jfcarmonag@gmail.com}

\author{Kevin D\'avila}
\address{Fundaci\'on Universidad del Norte, Km.5 V\'ia Puerto Colombia, Barranquilla, Colombia}
\email{kevindavilacastellar@gmail.com}

\author{Alf Onshuus}
\address{Universidad de los Andes,
Cra 1 No 18A-10, Bogot\'{a}, Colombia}
\email{aonshuus@uniandes.edu.co}

\author{Rafael Zamora}
\address{Escuela de Matem\'atica, Universidad de Costa Rica, San Pedro, San Jose, Costa Rica}
\email{rafael.zamora\_c@ucr.ac.cr}

\begin{document}

\maketitle

\date{}

\begin{abstract}
We prove an analogue of the fixed-point theorem for the case of definably amenable groups.
\end{abstract}

\section{Introduction}

In the context of locally compact groups, there are several
properties that have been proved to be equivalent to
amenability\footnote{In this paper we say that a topological group
is amenable if it admits a left invariant finitely additive
probability measure on the Borel subsets.}. One of those is the
fixed-point property (a generalization of the Markov-Kakutani
theorem), which states that any affine continuous action over a
compact convex subset of a locally convex vector space has a fixed
point.

The analogue of amenability for definable groups was defined by
\cite{Hr1} as follows: A definable group is {\it definably
amenable} if there is a left-invariant finitely-additive
probability measure on all its definable subsets (a left invariant
{\it global Keisler Measure}). It is known that every stable group
is definably amenable \cite{Po1}, and that groups definable in a
dependent theory are definably amenable if and only if they have
an $f$-generic type \cite{Hr1} (called strongly $f$-generic types
in \cite{Ch1}).

In this note we study definably amenable groups and prove that
definable amenability is equivalent to a fixed-point condition. We
do not assume that the group is in any ``good'' class in the sense
of Classification Theory.

The paper is divided as follows: In Section \ref{sec1}. we study
the $\sigma$-topology of a definable group, which allows us to
study the $\sigma$-algebra generated by definable sets (and
measures on it) using topological tools. The notion of
$\sigma$-topology allows us to see definable sets and topological
spaces as objects of the same category. In particular, we prove
that the $\sigma$-continuous image of a definable subset of an
$\omega_1$-saturated structure in a Polish space is compact.

In \ref{sec2} we show that the existence of an invariant mean is
equivalent to the existence of an invariant measure, when the
group is $\omega_1$-saturated. This would be an immediate
corollary of Hahn-Kolmogorov Theorem.

\newpage

In Section \ref{sec4} we prove the main result of the paper:

\begin{theorem*}
The following are equivalent:

\begin{enumerate}
\item $G$ is an $\omega_1$-saturated definably amenable group.

\item If $\pi$ is a linear action of $G$ into a convex compact subset $Y$
of a locally convex topological vector space $V$ such that
\begin{enumerate} \item $\pi_g$ is continuous for every $g\in G$.
\item $\pi_k$ is $\sigma$-continuous for some $k\in Y$.
\end{enumerate}
Then $\pi$ has a fixed point.
\end{enumerate}

\end{theorem*}

When working in with the definable topology $\sigma$-continuity is equivalent to logic-continuity (see Theorem \ref{continuity}), a concept which
may be more familiar for model theorists. We
state everything in terms of $\sigma$-continuity because it allows for cleaner proofs.

\section{The $\sigma$-topology of a first-order structure}\label{sec1}

In this section we introduce the notion of $\sigma$-space, which
will allow us to consider definable sets in first-order
structures, open sets in topological spaces, and even Borel
measurable sets in a $\sigma$-algebra, as objects of the same
category.

\subsection{The category of $\sigma$-spaces}

\begin{definition}
Let $X$ be a non-empty set, we say that $\tau\subset \mathcal{P}(X)$ is a {\it $\sigma$-topology} on $X$ if:
\begin{enumerate}
\item $\emptyset$, $X$ are in $\tau$.
\item $\tau$ is closed under finite intersections and {\it countable} unions.
\end{enumerate}

We say that $(X,\tau)$ is a {\it $\sigma$-space} if $\tau$ is a $\sigma$-topology on $X$.
\end{definition}

The definitions of {\it $\sigma$-open}, {\it $\sigma$-closed} and  {\it $\sigma$-continuity} are the natural ones.

\begin{exm}
$\left.\right.$

\begin{itemize}

\item Any topological space is a $\sigma$-space.

\item Any $\sigma$-algebra is a $\sigma$-space (although here there is no distinction between $\sigma$-open and $\sigma$-closed sets).

\end{itemize}
\end{exm}

\begin{dfn}
A collection $B$ of subsets of $X$ is a basis for a
$\sigma$-topology $\tau$ if every $U\in \tau$ is a countable union
of sets in $B$.
\end{dfn}

\begin{rmr}
The family $B$ is a basis for some $\sigma$-topology if and only if it is closed under finite intersections.
\end{rmr}

Let $M$ be any structure. If $X$ is an $A$-definable set in a
structure $M$, then $\Def_A(X)$ satisfies the conditions of the
previous remark and therefore is a basis for a $\sigma$-topology
on $X$.

From now on, we will consider any structure $M$ as a $\sigma$-space, where the $\sigma$-topology is generated by $\Def_M(M)$.

\begin{exm}
Let $\mathbb{R}^*$ be the non-standard reals (which is an $\omega_1$ saturated extension of $\mathbb{R}$) and let $st:\mathbb{R}^*\to \mathbb{R}\cap\{-\infty,\infty\}$ the standard map. Then $st$ is $\sigma$-continuous (seeing $\mathbb{R}^*$ as a first-order structure and $\mathbb{R}\cap\{-\infty,\infty\}$ as the usual topological space). Notice that the image of $\mathbb{R}^*$ is compact.
\end{exm}

\begin{rmr}
Any structure $M$ can be seen as a subset of its (Stone space)
space of types $S(M)$ (the Stone space of $\Def_M(M)$), where the
inclusion $i$ is given by $i(m)=tp(m/M)$. The topology of $S(M)$
is compact but is discrete when restricted to $M$, therefore the
inclusion is not $\sigma$-continuous in general.
\end{rmr}

We are interested in treating first-order structures as $\sigma$-spaces. The next definition will characterize saturation in terms of the $\sigma$-topology (this is merely a translation, but it would be very useful in order to characterize the image of saturated structures under $\sigma$-continuous functions).

\begin{definition}
We say that a $\sigma$-topology $(X,\tau)$ is {\it countably compact} if every countable open cover of $X$ has a finite subcover.
\end{definition}

By definition of saturation, we have the following:

\begin{rmr}
A structure $M$ is $\omega_1$-saturated if and only if it is countably compact as a $\sigma$-space.
\end{rmr}

The image of a countably compact space under a $\sigma$-continuous function is countably compact. Therefore we have the following result.

\begin{theorem}
If $X$ is countably compact, $Y$ is a second-countable topological space and $f:X\to Y$ is $\sigma$-continuous, then $f(X)$ is compact in $Y$.

\end{theorem}

\subsection{Definable Functions}
We remark that the notion of {$\sigma$-continuity} has some similarities with the notion of {\it definable function} given in \cite{Gi}:
\begin{definition}[Gismatullin, Penazzi, Pillay]
Let $M$ be any structure, let $X$ be an $M$-definable set in a saturated extension $M^*$ and $C$ a compact space. A function $f:X\to C$ is {\it definable} if for any closed set $D$, its pre-image is type-definable over $M$.
\end{definition}

It may be worth to clarify how they differ from each other.

Notice that the pre-image of closed sets by {\it definable} functions are sets that are type-definable over some small model $M$, while the pre-image of closed sets by $\sigma$-continuous functions are type-definable over the whole model, but only with countably many formulas. It is clear then that $tp:X\to S_M(X)$ is a definable function but it is not $\sigma$-continuous in general.

Our definition comes from the fact that we need to deal with the $\sigma$-algebra generated by {\it all} the definable sets of a group, so we cannot relativize to a small model. On the other hand, the topology generated by the definable sets is the discrete one and the notion of {\it definable} function, without relativizing to a small model, is no longer useful. Therefore, we need to take a weak notion of topology (i.e $\sigma$-topologies) and to study $\sigma$-continuity instead.

\subsection{Measures}\label{sec2}

The $\sigma$-algebra generated by the definable sets of a structure $M$ is precisely the {\it Borel} subsets of $M$ when is seen as a $\sigma$-space.

\begin{definition}
Let $M$ be any structure, a {\it Keisler measure on $M$} is a finitely-additive probability measure $\mu_0$ over $\Def(M)$.
\end{definition}

We recall here the Hahn-Kolmogorov Theorem, which will allow us to extend $\mu_0$ to a $\sigma$-additive measure over the $\sigma$-algebra generated by $\Def(M)$.

\begin{theorem}[Hahn-Kolmogorov] If $X$ has a finitely-additive probability measure $\mu_0$ over an algebra of subsets $B$, such that:

For every countable family $\{D_i\}_{i<\omega}$ of disjoint subsets in $B$ such that $\bigcup_{i<\omega} D_i \in B$, we have that $\mu_0(\bigcup_{i<\omega} D_i) = \sum_{i<\omega} \mu_0(D_i)$.

Then, there is a unique measure $\mu$ over the $\sigma$-algebra generated by $B$ that extends $\mu_0$.

\end{theorem}

\begin{theorem}\label{extension}
Every Keisler measure $\mu_0$ over an $\omega_1$-saturated
structure $M$ can be extended to a $\sigma$-additive measure over
the $\sigma$-algebra generated by $\Def(M)$.
\end{theorem}

\begin{proof}

We need to check that $(M,\Def(M),\mu_0)$ satisfies the hypothesis
of the Hahn-Kolmogorov Theorem: Let $\{D_i\}_{i<\omega}$ be a
family of disjoint subsets in $\Def(M)$ such that
$\bigcup_{i<\omega} D_i = D \in \Def(M).$ By $\omega_1$-saturation
only a finite number of $D_i$'s are non empty, hence the
conclusion follows.

\end{proof}

Let us note that if $M$ is $\omega_1$-saturated and has a Keisler
measure, then it may be seen as a measurable space, where the
$\sigma$-algebra is generated by the definable sets. Moreover, if
$X$ is a topological space with a Borel-measure, then every
$\sigma$-continuous function $f:M\to X$ is measurable in the usual
sense.

\subsection{Polish spaces and $\sigma$-continuity}

We establish several characterizations of $\sigma$-continuity over
Polish spaces. Recall that a Polish space may be characterized as
a second-countable locally-compact Hausdorff space.

\begin{theorem}\label{continuity}
Let $f:M\to X$ be a function from an $\omega_1$-saturated structure $M$ to a Polish-space $X$. The following are equivalent:
\begin{itemize}
\item $f$ is $\sigma$-continuous.
\item For every $K\subset U$ in $X$, with $K$ compact and $U$ open, there exists a definable set $D$ such that \[f^{-1}(K)\subset D \subset f^{-1}(U).\]
\item For every $x\in X$ and for every open $U$ containing $x$, there exists $K\subset U$ a compact
neighborhood of $x$ and $D$ definable such that
\[f^{-1}(K)\subset D\subset f^{-1}(U).\]
\end{itemize}
\end{theorem}

\begin{proof}
(1) $\Rightarrow$ (2): By $\sigma$-continuity, $f^{-1}(K)$ is $\sigma$-closed and $f^{-1}(U)$ is $\sigma$-open. Therefore, by compactness and $\omega_1$-saturation there must be a definable set in between.

(2) $\Rightarrow$ (3): This is due to the local-compactness of $X$.

(3) $\Rightarrow$ (1): Let $U$ be an open set. For every $x\in U$, we can find a neighborhood $K_i\subset U$ and $D_x$ definable such that $f^{-1}(K_x)\subset D_x\subset f^{-1}(U)$. Notice that $\{K_x\}$ is a cover of $U$ so we may find a countable subcover $\{K_i\}_{i<\omega}$ of $U$ (this cover exists because $X$ is Polish). By hypothesis, for every $i$ there exists $D_i$ definable such that $f^{-1}(K_i)\subset D_i \subset f^{-1}(U)$.

Clearly $f^{-1}(U)=\bigcup_{i<\omega} D_i$. \end{proof}

\begin{lemma}\label{addition}
If $f_1,f_2: M\to \mathbb{R}$ are $\sigma$-continuous, then both $(f_1,f_2):\G\to\mathbb{R}^2$  and $f_1+f_2$ are $\sigma$-continuous.
\end{lemma}

\begin{proof}

Let $U$ be an open set of $\mathbb{R}^2$ and let $\bar{x}=(x_1,x_2)\in U$. Take any box $U_1\times U_2$ of open sets such that $\bar{x}\in U_1\times U_2 \subset U$ and take $K_1\times K_2$ a box of compact sets where $K_i\subset U_i$ is a neighbourhood of $x_i$. By $\sigma$-continuity of $f_1$ and $f_2$ we can find $D_1$ and $D_2$ such that $f_1^{-1}(K_i)\subset D_i\subset f_2^{-1}(U_i).$ Therefore, $$(f_1,f_2)^{-1}(K_1\times K_2)\subset D_1\cap D_2\subset (f_1,f_2)^{-1}(U_1\times U_2)$$ and therefore $(f_1,f_2):\G\to\mathbb{R}^2$ is $\sigma$-continuous.

Now, notice that $$f_1+f_2: M \xrightarrow{(f_1,f_2)} \mathbb{R}^2\xrightarrow{+}\mathbb{R}.$$ Since the composition of $\sigma$-continuous functions is $\sigma$-continuous, we have that $f_1+f_2$ is $\sigma$-continuous.
\end{proof}

\begin{rmr}
If $D$ is a definable set, then its characteristic function
\[{\bf 1}_D:M\to \mathbb{R}\] is $\sigma$-continuous.
\end{rmr}

\begin{theorem}\label{Banach}
Let $M$ be a saturated structure. Then the space $C_{\sigma}(M)$ of $\sigma$-continuous
functions of $M$ over $\mathbb{R}$ equipped with the supremum norm is a Banach space.
\end{theorem}

\begin{proof}
By Lemma \ref{addition} we know that the $\sigma$-continuous functions
form a vector space. We just need to show that Cauchy sequences converge.

Let $\langle f_i\rangle_{i\in \omega}$ be a sequence of
$\sigma$-continuous functions from $M$ to $\mathbb R$ such that for
every $\epsilon$ there is an $N$ such that for all $x\in M$, for
all $i,j>N$ we have $|f_i(x)-f_j(x)|<\epsilon$. Since $\mathbb R$
is complete the function $F(x)=\lim_{i\rightarrow \omega} f_i(x)$ exists.

We need to show that $F$ is $\sigma$-continuous. Let
$U$ be an open set in $\mathbb R$ and take $V:=B_{\epsilon/8}(a)$ such that $B_{\epsilon}(a)\subset U$. By Theorem \ref{continuity}, it is enough to show that there is a definable $D$ such
that $f^{-1}(\overline{V})\subseteq D\subseteq F^{-1}(U)$.

Let $N$ be such that for $i,j\geq N$ we have
$|f_i(x)-f_j(x)|<\epsilon/8$ which implies that
$|f_N(x)-F(x)|\leq \epsilon/8$. By $\sigma$-continuity of $f_N$, let
$D$ be a definable set such that
\[
f_N^{-1}(\overline{B_{\epsilon/4}(a)})\subseteq D \subseteq
f_N^{-1}(B_{\epsilon/2}(a)).
\]

Now, if $x\in D$ we have $f_N(x)\in B_{\epsilon/2}(a)$ and
$|f_N(x)-F(x)|\leq \epsilon/8$ which by triangle inequality
implies $F(x)\in B_{5\epsilon/8}(a)$.

On the other hand, if $F(x)\in \overline{B_{\epsilon/8}(a)}$, again
by triangle inequality we have $f_i(x)\in B_{\epsilon/4}(a)$, hence $x\in D)$. It follows that
\[
F^{-1}(\overline{V})\subseteq D \subseteq F^{-1}(U),
\]
as required.
\end{proof}

\section{Definable amenability and the fixed-point property} \label{sec4}

\begin{definition}
A definable group $\G$ is {\it definably amenable} if it has a left-invariant Keisler measure.
\end{definition}

If $\G$ is an $\omega_1$-saturated definably amenable group, by
Theorem \ref{extension} we know that its Keisler measure $\mu_0$
can be extended to a measure $\mu$ on $\mathcal{B}$, the
$\sigma$-algebra generated by $\Def(M)$.  It is easy to see that
$\mu$ is left-invariant.

Notice that $\mu$ can be seen as a linear functional from
$\mathcal{S}(\G)$ to $\mathbb{R}$, where $\mathcal{S}(\G)$ is the
set of simple measurable functions from $\G$ to $\mathbb{R}$:

\[\mathcal{S}(\G):=\left\{\sum_{i\leq n} a_i {\bf 1}_{B_i}\mid B_i\in \mathbb{B},\, a_i\in \mathbb{R},\, n<\omega \right\}\]

It is well known that $\mathcal{S}(\G)$ is a dense subset of the
Banach space of essentially bounded measurable functions
$L^{\infty}(\G)$.

With this in mind, we can prove now the following theorem:

\begin{theorem}
If $\G$ is an $\omega_1$-saturated definably amenable group, then there exists an invariant mean on $L^{\infty}(\G)$.

\end{theorem}

\begin{proof}
Since $\mathcal{S}(\G)$ is dense in $L^{\infty}(\G)$, the left-invariant mean $\mu\in \mathcal{S}(\G)$ has a unique extension to a mean $\mu'\in L^{\infty}(\G)$. Notice that $\mu'({\bf 1}_{\G})=1$.  It only remains to show that $\mu'$ is invariant:

Let $f\in L^{\infty}(\G)$ and $g\in G$. By density, for every $\epsilon > 0$ there exists $f' \in \mathcal{S}(\G)$ such that $|f-f'|<\epsilon/2$. Therefore
\[|\mu'(_gf-f)|\leq |\mu'(_g(f-f'))|+|\mu'(_g f'-f')|+|\mu'(f'-f)|\]
\[\leq 2|f-f'|<\epsilon.\]\end{proof}

\begin{definition}
Let $\pi:\G\times X\to X$ be a group action. We say that $\pi$ is {\it separately $\sigma$-continuous} if \[\pi_g: X\to X;\, \pi_g(x)=\pi(g,x)\] and \[\pi_x:G\to X;\,\pi_x(g)=\pi (g,x)\] are $\sigma$-continuous for every $g\in \G$ and every $x\in X$.  This implies in particular that $\pi_g$ is an homeomorphism for every $g$.

\end{definition}

\begin{definition}
Let $V$ be a vector space, $F$ a family of subsets of $V$, and $U\subset V$. We say that $F$ is {\it $U$-fine} if every element of $F$ is contained in a translate of $U$.
\end{definition}

\begin{theorem}[Fixed-point theorem]\label{fixedpoint}
The following are equivalent:

\begin{enumerate}
\item $G$ is an $\omega_1$-saturated definably amenable group.

\item If $\pi$ is a linear action of $G$ into a convex compact subset $Y$
of a locally convex topological vector space $V$ such that
\begin{enumerate} \item $\pi_g$ is continuous for every $g\in G$.
\item $\pi_k$ is $\sigma$-continuous for some $k\in Y$.
\end{enumerate}
Then $\pi$ has a fixed point.
\end{enumerate}
\end{theorem}

\begin{proof}
(1)$\rightarrow$(2): This is very similar to the proof in
\cite{Wa}. Let $k\in Y$ such that $\pi_k$ is $\sigma$-continuous.
Let
\newline $F=\{U_1,...,U_n\}$ be a minimal open cover of $Y$,
define
\[S_1=\pi_k^{-1}(U_1),...,S_i=\pi_k^{-1}(U_i\setminus
(U_1\cup...\cup U_{i-1}))\dots S_n=\pi_k^{-1}(U_n\setminus
(U_1\cup...\cup U_{n-1}))\] and $\mu_{U_i}=\mu(S_i)$ for every
$i\leq n$.

The numbers $\mu_{U_i}$ depend on the order of $F$, however $\Sigma_{i}\mu_{U_i}=1$ for any order of $F$.

Let $D$ be the set of finite minimal open covers of $Y$ ordered by refinement.

For every $F \in D$, let $\left\{s_U^F\right\}_{U \in F}$ be a set of
points such that $s_U^F \in U \cap Y$, and define the function.

\begin{equation*}
\begin{split}
\Phi : D &\longrightarrow Y \\
F & \mapsto \sum_{U \in F} \mu_U \cdot s_{U}^F,
\end{split}
\end{equation*}

The expression defining $\Phi(G)$ is a convex
linear combination, hence an element of $Y$.\\

\begin{claim} If $F_1 \preceq F_2$ are covers in $D$, $V$ is a symmetric convex neighborhood of 0 and $F_1$ is $V/2$-fine, then
$\Phi(F_1)-\Phi(F_2) \in V$.
\end{claim}

\begin{proof} We only show the case when $F_1 = \left\{U\right\}$, other cases will be similar:

Since $F_1$ is $V/2$-fine, there is some element $t$ such that $t+V/2 \supseteq U$. Let $\Phi(F_1) = s_U^{F_1}$ and $\Phi(F_2) =
\sum_{U' \in F_2} \alpha_{U'}\cdot s_{U'}^{F_2}$, then we have
$$\Phi(F_1)-\Phi(F_2) = s_{U}^{F_1} - \sum_{U' \in F_2}\alpha
_{U'} \cdot s_{U'}^{F_2}$$
$$ = \left\{s_{U}^{F_1}-t \right\}-\left\{ \left(\sum_{U' \in F_2}
\alpha_{U'} \cdot s_{U'}^{F_2}\right)-t \right\},$$
but the two terms within braces belong to $V/2$. By symmetry and convexity, this difference is in $V$.
\end{proof}

\begin{claim}
$\Phi$ is a convergent net.
\end{claim}

\begin{proof}
We can assume that $\Phi(D)$ is infinite. By compactness of $Y$ we
know that $\Phi(D)$ has an accumulation point $\tilde{k}$. We will
show that $\tilde{k}$ is the limit of the net: let $W$ a symmetric
convex neighborhood of $0$ and let $F\in D$ be $W/4$ fine. Since
$\tilde{k}$ is an accumulation point, we may assume that
$\Phi(F)\in \tilde{k}+W/2$. By the previous claim, for every
$F\preceq F'$ we have that $\Phi(F)-\Phi(F')\in W/2$, therefore
$\Phi(F')\in \tilde{k}+W$. Therefore $\tilde{k}$ is a limit of the
net. By Hausdorff, it is unique.
\end{proof}

\begin{claim} \label{claim 2}
The limit of $\Phi$ does not depend on the choice of the points
$s_U^F$.
\end{claim}

\begin{proof}
let $\Psi$ be a function which is defined in the same way as
$\Phi$ but choosing different points $r_U^F$ for the covers in
$D$. Assume that $F$ is $W/4$-fine for an arbitrary convex
symmetric neighborhood of the origin $W$ and such that $\Phi(F)
\in \tilde{k} + W/2$, it follows then that $s_U^F - r_U^F \in W/2$
which in turns yields to $\Phi(F) - \Psi(F) \in W/2$, thus leading
to $\Psi(F) \in \tilde{k} + W$, as desired.
\end{proof}

\medskip

Now, since the action provided by each element of $G$ over $X$ is continuous,
we know that $\hat{g}: D \rightarrow D$ with $\hat{g}(F) := \left\{\pi_g(U) \mid
U \in F\right\}$ defines a permutation of $D$ for every $g \in G$. The action
of each element $g$ also permutes selection points for the open covers, this
allows us to define the following function.

\begin{equation*}
\begin{split}
\Phi_g : D &\longrightarrow Y\\
F &\mapsto \sum_{U \in \hat{g}(F)} \mu_U \cdot \left(\pi_g
\left(s_{g^{-1}(U)}^F\right)\right).
\end{split}
\end{equation*}

\noindent Claim \ref{claim 2} implies that $\Phi_g \rightarrow_{F}
\tilde{k}$.

On the other hand we have
$$\sum_{U \in \hat{g}(F)} \mu_U \cdot \left(\pi_g\left( \ s_{g^{-1}(U)}^F \right)
 \right) = \pi_g \left(\sum_{U \in \hat{g}(F)} \mu_U \cdot
 \ s_{g^{-1}(U)}^F\right) = \pi_g \left(\Phi(F)\right),$$

\noindent showing that $g\Phi \rightarrow_F \tilde{k}$\footnote{
Notice that if we fix the order on $F$ established in the definition of $\Phi$,
when passing to $gF$ the associated values $\mu_U$ do not change  since
the given measure is $G$-invariant.}. Finally, the continuity of the action of $g$
implies that $g\Phi \rightarrow_F g\tilde{k}$ showing that $\tilde{k}$ is in
fact a fixed point for the action.

\bigskip

(2) $\rightarrow$ (1): By Theorem \ref{Banach} we know that the
space $C_{\sigma}(G)$ of bounded $\sigma$-continuous functions
form a Banach vector space. Let $X:=C_{\sigma}(G)^*$ be its dual.
By Alaoglu, the unit ball is a convex compact subset of $X$ with
the weak$^*$ topology, and we can define a linear action of $G$ on
$X$ by having, for any $g\in G$, any $F\in X$, and any $f\in
C_{\sigma}(G)$,
\[
\pi_g(F)(f)=F(_gf),
\]
where \[_gf(h)=f(gh)\] for any $h\in G$.
\begin{claim}
For a fixed $g$ we have $\pi_g:X\rightarrow X$ is continuous in
the weak$^*$-topology.
\end{claim}

\begin{proof}
Let $U$ be a sub-basic open set in $X$ with the weak$^*$ topology,
so that $U:=\{F\mid F(f)\in B\}$ where $B$ is an open interval in
$\mathbb R$. Then
\[
\pi_g^{-1}(U)=\{F\mid _gF(f)\in B\}=\{F\mid F(_gf)\in B\} \]
which is by definition another sub-basic open set in $X$.
\end{proof}

Let $h\in G$ be any element. It is easy to see that the
evaluation map $F_h(f)=f(h)$ is in $X$.

\begin{claim}
For every $h\in G$, the map $F_h$ is $\sigma$-continuous.
\end{claim}

\begin{proof}
Once again, let $U$ be a sub-basic open set in $X$ of the form
\[U:=\{F\mid F(f)\in B\}\] for $B$ an open interval in $\mathbb R$.
By definition,
\[
F_h^{-1}(U):=\{g\in G \mid F_h(g)\in B\}=\{g\in G \mid f(gh)\in B\}.
\]
Since $f$ is $\sigma$-continuous we can find a countable family $\{D_i\}_{i<\omega}$ of definable subsets
\[
\bigcup_{i\in \omega} D_i=f^{-1}(U),\] and
\[
F_h^{-1}(U)=\left(f^{-1}\left(U\right)\right)h^{-1}=\left(\bigcup_{i\in
\omega} D_i\right)h^{-1}=\bigcup_{i\in \omega} D_ih^{-1}\] which
is a countable union of definable sets, as required.
\end{proof}

It follows that the action of $G$ on the unit ball of $X$ has a
fixed point $F$. The characteristic functions of
definable sets are $\sigma$-continuous, so we know that we can
define a finitely additive measure on definable subsets of $G$ by
$\mu(D)=F({\bf 1}_D)$. But then
\[
\mu(gD)=F({\bf 1}_{gD})=F(_g{\bf 1}_D)=\,_gF({\bf 1}_D)=F({\bf 1}_D)
\]
by invariance, thus proving definable amenability of $G$.

\end{proof}

\bigskip

To prove (1)$\rightarrow$(2) we only used that $\pi_k$ was
measurable. Since any $\sigma$-continuous function is a measurable
function, we have the following equivalent statement.

\begin{theorem*}
The following are equivalent:

\begin{enumerate}
\item $G$ is an $\omega_1$-saturated definably amenable group.

\item If $\pi$ is a linear action of $G$ into a convex compact subset $Y$
of a locally convex topological vector space $V$ such that
\begin{enumerate} \item $\pi_g$ is continuous for every $g\in G$.
\item For some $k\in Y$ the function $\pi_k$ is a measurable function from the $\sigma$-algebra of definable sets in $G$ into the Borel $\sigma$-algebra of $Y$.
\end{enumerate}
Then $\pi$ has a fixed point.
\end{enumerate}

\end{theorem*}

\end{document}